\documentclass{lms}

\usepackage{amsmath}
\usepackage{amsfonts}
\usepackage{enumerate}
\usepackage{url}
\usepackage{booktabs}
\usepackage[group-minimum-digits=4,group-separator=\text{,}]{siunitx}
\usepackage{mathtools}

\newtheorem{theorem}{Theorem}[section] 

\newtheorem{proposition}[theorem]{Proposition}

\newcommand{\QQ}{\mathbf{Q}}
\newcommand{\FF}{\mathbf{F}}
\newcommand{\ZZ}{\mathbf{Z}}
\newcommand{\PP}{\mathbf{P}}

\newcommand{\MM}{\mathsf{M}}
\newcommand{\pp}{\mathfrak{p}}

\newcommand{\norm}[1]{\Vert #1 \Vert}
\newcommand{\Jac}{\mathrm{Jac}}

\newnumbered{assertion}{Assertion}    
\newnumbered{conjecture}{Conjecture}  
\newnumbered{definition}{Definition}
\newnumbered{hypothesis}{Hypothesis}
\newnumbered{remark}{Remark}
\newnumbered{note}{Note}
\newnumbered{observation}{Observation}
\newnumbered{problem}{Problem}
\newnumbered{question}{Question}
\newnumbered{algorithm}{Algorithm}
\newnumbered{example}{Example}
\newunnumbered{notation}{Notation} 

\title[Computing $L$-series of certain genus three curves]
 {Computing $L$-series of geometrically hyperelliptic curves \\ of genus three} 

\author[D. Harvey, M. Massierer, and A.V. Sutherland]{David Harvey, Maike Massierer, and Andrew V. Sutherland}

\classno{11G20 (primary), 11Y16, 11M38, 14G10 (secondary)}

\extraline{
The first and second authors were supported by the Australian Research Council (DP150101689).
The third author was supported by NSF grants DMS-1115455 and DMS-1522526.
}

\begin{document}
\maketitle

\begin{abstract}
Let $C/\QQ$ be a curve of genus three, given as a double cover of a plane conic.
Such a curve is hyperelliptic over the algebraic closure of $\QQ$, but may not have a hyperelliptic model of the usual form over $\QQ$.
We describe an algorithm that computes the local zeta functions of $C$ at all odd primes of good reduction up to a prescribed bound $N$.
The algorithm relies on an adaptation of the ``accumulating remainder tree'' to matrices with entries in a quadratic field.
We report on an implementation, and compare its performance to previous algorithms for the ordinary hyperelliptic case.
\end{abstract}

\section{Introduction} 
\label{sec:intro}

Let $C/\QQ$ be a curve of genus three.
For an odd prime $p$ of good reduction for $C$, let $C_p$ denote the reduction of $C$ modulo $p$.
The zeta function of $C_p$ is defined by
\begin{equation}
\label{eq:zeta}
 Z_p(T) \coloneqq \exp\left(\sum_{k=1}^\infty \frac{\#C_p(\FF_{p^k})}k T^k \right) = \frac{L_p(T)}{(1-T)(1-pT)}.
\end{equation}
By the Weil Conjectures for curves, the numerator is of the form
 \[ L_p(T) = 1 + a_1 T + a_2 T^2 + a_3 T^3 + p a_2 T^4 + p^2 a_1 T^5 + p^3 T^6 \in \ZZ[T], \]
and has reciprocal roots of complex absolute value $p^{1/2}$.
In this paper we are interested in algorithms for computing $Z_p(T)$, for all good primes $p$ up to a prescribed bound $N$.

A closely related problem is to compute the first $N$ terms of the $L$-series associated to~$C$.
This is defined by formally expanding the Euler product
 \[ L(C,s) = \prod_p L_p(p^{-s})^{-1} = \sum_{n \geq 1} c_n n^{-s}. \]
To compute $c_n$ for all $n < N$, one must compute $a_1$ for all $p < N$, but one needs $a_2$ only for $p < N^{1/2}$ and $a_3$ only for $p < N^{1/3}$.
Note that for the primes of bad reduction, $L_p(T)$ is not necessarily given by \eqref{eq:zeta}, and must be computed by other means;
this problem is not addressed in this paper.


Curves of genus three over $\QQ$ come in two flavors, depending on the behavior of the canonical embedding $\phi : C \to \PP^2$.
The first possibility is that $\phi$ is a two-to-one cover of a plane conic~$Q$, in which case $C$ is geometrically hyperelliptic.
Otherwise, $\phi$ is an isomorphism from $C$ to a smooth plane quartic defined over $\QQ$, and we are in the nonhyperelliptic case.

In the hyperelliptic case, if $Q$ possesses a $\QQ$-rational point, then $Q$ is isomorphic to $\PP^1$ over~$\QQ$, and this yields a model for $C$ of the form $y^2 = h(x)$, with $h \in \ZZ[x]$.
For such curves, the first author proposed an algorithm that computes $L_p(T)$ for all odd $p < N$ using a total of $N (\log  N)^{3+o(1)}$ bit operations \cite{Harvey:HyperellipticPolytime}.
The average time per prime is thus $(\log N)^{4+o(1)}$.
Although that algorithm has not been implemented in full generality, the first and third authors \cite{HS:HyperellipticHasseWitt,HS-hassewitt2} have developed a simpler and closely related algorithm for computing the Hasse--Witt matrices~$W_p$ of the reductions modulo primes $p < N$ of a fixed hyperelliptic curve $y^2 = h(x)$ in average time $(\log N)^{4+o(1)}$ per prime; in practice, for curves of genus two and three this yields enough information to quickly deduce the local zeta functions.
This implementation outperforms existing packages, based on older algorithms, by several orders of magnitude.
An analogous algorithm for the nonhyperelliptic genus three case is currently under development and will be presented in a forthcoming paper.

The main contribution of this paper is an analogue of the algorithm of \cite{HS-hassewitt2} for the geometrically hyperelliptic case, without the assumption that $Q(\QQ) \neq \emptyset$.
The new algorithm takes as input an integer $N$ and homogeneous polynomials $f, g \in \ZZ[X,Y,Z]$, with $\deg g = 2$ and $\deg f = 4$, specifying the curve
\begin{equation}
\label{eq:eqn}
 g(X,Y,Z) = 0, \qquad w^2 = f(X,Y,Z).
\end{equation}
The equation $g = 0$ defines the conic $Q$, and $w^2 = f$ describes the two-to-one cover.
The output of the algorithm is the sequence of polynomials $L_p(T)$ for all good primes $p < N$.

The new algorithm is mainly intended for use in the case that $Q(\QQ) = \emptyset$.
However, the algorithm works perfectly well if $Q(\QQ) \neq \emptyset$; this may be useful, for example, if a $\QQ$-rational point on $Q$ exists but cannot be determined efficiently due to the difficulty of factoring the discriminant of $g$.
Of course, if a $\QQ$-rational point is known, then it may be profitable to switch to a standard hyperelliptic model $y^2 = h(x)$ and apply the algorithm of \cite{HS-hassewitt2} instead.

Our focus is on designing a \emph{practical} algorithm: we want to actually compute local zeta functions on real hardware for values of $N$ that are as large as is practical.
From a theoretical point of view, the existence of a complexity bound analogous to \cite{Harvey:HyperellipticPolytime} for curves of the type \eqref{eq:eqn} was essentially demonstrated by the first author in \cite{Har-arithzeta}.
The result may be stated as follows.
For any polynomial $F$ with integer coefficients we denote by $\norm{F}$ the maximum of the absolute values of the coefficients of $F$.
\begin{theorem}
There exists an explicit deterministic algorithm with the following properties.
The input consists of an integer $N \geq 2$, and polynomials $f$ and $g$ describing a genus three curve $C$ as in \eqref{eq:eqn}.
The output is the collection of $L_p(T)$ associated to $C$ for all good primes $p < N$.
The algorithm runs in
 $N \log^2 N \log^{1+o(1)} (N \norm f \norm g)$
bit operations.
\end{theorem}
We omit the details of the proof.
Ignoring the dependence on $\norm f \norm g$, the complexity bound is a special case of \cite[Theorem~1.1]{Har-arithzeta}, which applies to \emph{any} fixed variety over $\ZZ$.
To get the right dependence on $\norm f \norm g$, one may invoke \cite[Theorem~1.4]{Har-arithzeta} and apply the ``inclusion-exclusion trick'' of \cite[\S3]{LW-counting} (see the proof of \cite[Theorem~1.1]{Har-arithzeta} for a similar argument).
The difficulty with the algorithm just sketched is that the implied big-$O$ constant is enormous, essentially because the algorithms of \cite{Har-arithzeta} are designed for maximum possible generality.
To obtain a practical algorithm we must exploit the geometry of the situation at hand.

Our strategy is motivated by the following observation.
If we only want to compute $L_p(T)$ for a \emph{single} prime $p$, we may start by finding some $\FF_p$-rational point on the conic (such a point exists for all odd $p$).
This leads to a rational parametrization for the conic over $\FF_p$ and hence a model for $C_p$ of the form $y^2 = h(x)$ over $\FF_p$.
We may then apply any of the known point-counting algorithms for hyperelliptic curves over finite fields.
To mount a global attack along these lines, we must somehow choose these $\FF_p$-rational points ``coherently'' as $p$ varies.
This cannot be done over $\QQ$, because we are expressly avoiding any assumptions about $\QQ$-rational points on the conic.
On the other hand, it is easy to construct a quadratic extension $K = \QQ(\sqrt D)$ for which~$Q$ has $K$-rational points.
We may then parametrize $Q$ over $K$ to obtain a model $y^2 = h(x)$ of a hyperelliptic curve $C'/K$ that is isomorphic to the base change of $C$ to $K$, where $h\in \mathcal O_K[x]$.
Of course the curve $C'$ is not isomorphic to the original curve over $\QQ$ (it is not even defined over $\QQ$), but it nevertheless retains much arithmetic information about the original curve.

This is exactly the approach we take in this paper.
We start in Section \ref{sec:model} by explaining how to construct an appropriate field $K$ and a model $y^2 = h(x)$ for $C'$ over $K$.
In Section~\ref{sec:recurrences}, we set up recurrences for computing coefficients of powers of $h(x)$, analogous to \cite{HS-hassewitt2}, and in Section \ref{sec:tree} we show how to solve these recurrences efficiently by means of an ``accumulating remainder tree'' for matrices defined over a quadratic field.
Section \ref{sec:hw} applies these techniques to the problem of computing the Hasse--Witt matrices associated to $C'$, which in turn leads in Section \ref{sec:lpolys} to information about $L_p(T) \pmod p$.
Finally, to pin down $L_p(T) \in \ZZ[T]$ we perform a baby-step/giant-step search in the Jacobian of the curve; this is discussed in Section~\ref{sec:lifting}.
Section \ref{sec:summary} presents a complete statement of the algorithm, and the last section reports on an implementation and gives some performance data.

We will not give a formal complexity analysis of the algorithm; instead, we will discuss complexity issues as they arise, with an eye towards practical computations.
From an asymptotic perspective, our algorithm to compute $L_p(T) \in \ZZ[T]$ does \emph{not} run in average polynomial time, because the lifting step (see Section \ref{sec:lifting}) uses $p^{1/4+o(1)}$ bit operations per prime.
Nevertheless, as demonstrated by the timings in Section \ref{sec:performance}, the cost of the lifting step is negligible over the range of our experiments, and, by extrapolation, over the range of all currently feasible computations.
Moreover, the lifting step is trivially parallelizable (the rest of the algorithm is not), so this is unlikely to ever be a problem in practice.

There are two main applications of this new class of ``average polynomial time'' algorithms.
The first is the investigation of higher-genus variants of the Sato--Tate conjecture.
The original Sato--Tate conjecture proposed that for a fixed elliptic curve over $\QQ$, the distribution of the polynomials $L_p(T)$ (suitably normalized) obeys a particular statistical law when sampled over increasing values of $p$.
This is now a theorem thanks to work of Richard Taylor and collaborators \cite{CHT-automorphy,HST-automorphy,Tay-automorphy}, but analogues for curves of higher genus remain open.
The last few years have seen significant progress on understanding the details of the genus two case \cite{FKRS:SatoTate,Gonz:Modular,Joh:SatoTate,KatzSarnak:RandomMatrices,KS:SatoTate}, and attention is now shifting to genus three \cite{FS-genus3families,LS-picard}.
Briefly, the role of these algorithms is to assist in identifying potential candidate curves possessing certain Sato--Tate groups, by computing corresponding Sato--Tate statistics (moments of the sequence of normalized $L$-polynomials) for each of a large set of candidates.
The algorithm described in the present paper will be used to investigate the possibility that certain Sato--Tate distributions in genus~3 are encountered only for curves of the form~\eqref{eq:eqn}.

The second application is computing zeros and special values of $L$-functions to high precision; this played an important role in the recent addition of genus 2 curves to the $L$-functions and Modular Forms Database (LMFDB) \cite{lmfdb}, as described in \cite{BSSVY:Genus2DB} (as noted above, this application also requires the Euler factors at primes of bad reduction).

\smallskip
{\it Notation.}
We denote by $\MM(s)$ the number of bit operations required to multiply $s$-bit integers.
We may take $\MM(s) = s (\log s)^{1+o(1)}$ \cite{Furer2009,HLvdH-zmult,SS:IntegerMultiplication}.
As in \cite{HS-hassewitt2}, we assume that $\MM(s)/(s \log s)$ is increasing, and that the space complexity of $s$-bit integer multiplication is $O(s)$.

\section{Constructing a suitable quadratic field and hyperelliptic model}
\label{sec:model}

Let $C$ be a genus three curve over $\QQ$ as in \eqref{eq:eqn}.
The goal of this section is to construct an integer $D$, not a square and not divisible by $4$, and a squarefree polynomial $h \in O_K[x]$, where $O_K$ is the ring of integers of $K = \QQ(\sqrt D)$, such that $y^2 = h(x)$ is a model for $C' = C \times_\QQ K$ (the base change of $C$ to $K$).
Moreover, we require that $\deg h = 8$ and that $h(0) \neq 0$.

We assume that elements of $O_K = \ZZ[\alpha]$ are represented by pairs of integers corresponding to the coefficients of $1$ and $\alpha$, where $\alpha = \sqrt D$ if $D \equiv 2, 3 \pmod 4$, or $\alpha = \frac12(1 + \sqrt D)$ if $D \equiv 1 \pmod 4$.
In our applications we take $D$ to be squarefree, but this is not strictly necessary.

Choose any line $L$ in $\PP^2$ defined over $\QQ$, say $X = 0$.
The points of intersection of $L$ and $Q$ are defined over an extension $K = \QQ(\sqrt D)$ for some $D \in \ZZ$.
Note that $D$ is obtained as the discriminant (possibly adjusted by some square factor) of a quadratic equation obtained by solving $g = 0$ simultaneously with the equation of~$L$.
Let $P_0 \in Q(K)$ be one of the intersection points (of which there are at most two).
Now take a second line $L'$ in $\PP^2$, also defined over~$\QQ$, which does not contain $P_0$.
By projection from $P_0$, we obtain a $K$-rational parametrization of $Q(K)$ by the points of $L'(K)$.
Taking $x$ to be a coordinate for some affine piece of $L'$, we may write the parametrization as $(\psi_1(x), \psi_2(x), \psi_3(x)) \in \PP^2$, where the $\psi_i \in O_K[x]$ are polynomials of degree at most two.
Our preliminary model for $C'$ is then $y^2 = h(x)$, where $h(x) = f(\psi_1(x), \psi_2(x), \psi_3(x))$.

If $D$ is a square, then $K = \QQ$ and we have actually found a $\QQ$-rational point on $Q$.
In this case we could now simply apply the algorithm of \cite{HS-hassewitt2} to the equation $y^2 = h(x)$.
For the remainder of the paper, we assume that $D$ is not a square, so that $K/\QQ$ is a quadratic extension (although in fact the algorithm still works, \emph{mutatis mutandis}, for square $D$).

Clearly $\deg h \leq 8$.
Note that $C'$ is isomorphic to $C$ over $K$ so it must have genus three; hence $\deg h \geq 7$ and $h$ is squarefree.
It remains to enforce the conditions that $\deg h = 8$ and that $h(0) \neq 0$.
If $h(0) = 0$ we may replace $h(x)$ by $h(x-c)$ where $c$ is a small integer with $h(c) \neq 0$.
If $\deg h = 7$ we can replace $h(x)$ by $x^8 h(1/x)$, and translate again.
These transformations all correspond to birational maps.
In this way we obtain a model for $C'$ with $\deg h = 8$ and $h(0) \neq 0$.

Note that the conditions $\deg h = 8$ and $h(0) \neq 0$ are imposed only to simplify the presentation later.
From a complexity point of view it is actually \emph{better} to have $\deg h = 7$, or $h(0) = 0$, or both.
These occur when the curve has Weierstrass points defined over $K$, and in these cases we can work with smaller recurrence matrices in Section \ref{sec:recurrences}; see \cite[Section~6.2]{HS-hassewitt2} for details.
Our current implementation always assumes that $\deg h = 8$ and that $h(0) \neq 0$.

The running time of the main algorithm is quite sensitive to the bit size of the coefficients of $h$, and to some extent the bit size of $|D|$.
In the procedure described above, we have made no attempt to minimize these quantities.
If this became a bottleneck, one could try changing variables to obtain a conic with smaller coefficients \cite{CR-conics}, and one can also attempt to reparametrize $L'$ to minimize the coefficients of the resulting $h(x)$ \cite{SC-reduction}.
We do not know if these methods would lead to optimal running times; this seems to be a difficult problem, because of the dependence of the hyperelliptic model on the choice of $D$.
We suspect that to obtain a truly optimal model, one would need to optimize $D$ and $h(x)$ simultaneously.
In any case, if we restrict our attention to certain very simple conics, then we can often write down parametrizations for which the bit sizes remain under control;
see Section \ref{sec:performance} for an example.
We expect that this will be sufficient for the application to the Sato--Tate conjecture.

\section{Recurrences for the hyperelliptic model}
\label{sec:recurrences}

Let $y^2 = h(x)$ be a model for $C'$ over $K = \QQ(\sqrt D)$ as in Section~\ref{sec:model}.
For each odd prime~$p$ we define a row vector $U_p \in (O_K/p)^3$ by
 \[ (U_p)_j \coloneqq h^{(p-1)/2}_{p-j} \bmod p, \qquad j \in \{1, 2, 3\}. \]
Here $h^{(p-1)/2}_{p-j}$ denotes the coefficient of $x^{p-j}$ in $h^{(p-1)/2}$.
Note that $O_K/p$ is not necessarily a field, because $p$ may split in $K$.

These vectors are closely related to the Hasse--Witt matrices for $C$, which are in turn related to the local zeta functions.
The exact relationship is discussed in Sections \ref{sec:hw} and \ref{sec:lpolys}.
In this section we concentrate on the following problem: given a bound $N$, compute $U_p$ for all odd $p < N$ (except a small number of ``exceptional'' primes as indicated below).

Write $h(x) = h_0 + h_1 x + \cdots + h_8 x^8$ where $h_i \in O_K$, $h_0 \neq 0$.
For each integer $k \geq 1$, define an $8 \times 8$ matrix $M_k$ with entries in $O_K$ by
 \[
M_k \coloneqq \begin{bmatrix}
      0 &  \cdots  &   0     & (8-2k)h_8 \\
2 k h_0 &  \cdots  &   0     & (7-2k)h_7 \\
 \vdots &  \ddots  & \vdots  & \vdots    \\
      0 &  \cdots  & 2 k h_0 & (1-2k)h_1
\end{bmatrix}.
\]
Also define the vector $V_0 = [0,0,0,0,0,0,0,1] \in (O_K)^8$.

\begin{proposition}
\label{prop:Up}
Let $p$ be an odd prime with $(h_0, p) = 1$.
Then $U_p$ is equal to the vector consisting of the last three entries (in reversed order) of the vector
 \[ \frac{-1}{h_0^{(p-1)/2}} V_0 M_1 \cdots M_{p-1} \pmod p. \]
\end{proposition}
\begin{proof}
For $1 \leq k \leq p-1$ let
 $ v_k \coloneqq [h^{(p-1)/2}_{k-7}, \ldots, h^{(p-1)/2}_k] \in (O_K/p)^8$.
Using exactly the same argument as in \cite[Section 2]{HS-hassewitt2}, one may show that $v_k$ satisfies the recurrence
\begin{equation}
\label{eq:vk}
 v_k = \frac{1}{2 k h_0} v_{k-1} M_k \pmod p.
\end{equation}
Iterating this recurrence yields
 \[ v_{p-1} = \frac{1}{(p-1)! (2 h_0)^{p-1}} v_0 M_1 \cdots M_{p-1} \pmod p. \]
Since $2^{p-1} (p-1)! = -1 \pmod p$ and $v_0 = [0, \ldots, 0, (h_0)^{(p-1)/2}] \pmod p$, we have
 \[ v_{p-1} = \frac{-1}{h_0^{(p-1)/2}} V_0 M_1 \cdots M_{p-1} \pmod p. \]
The last three entries of $v_{p-1}$ are precisely the entries of $U_p$.
\end{proof}

According to the proposition, the problem of computing $U_p$ for all odd primes $p < N$, except those for which $(h_0, p) \neq 1$, reduces to the problem of computing $V_0 M_1 \cdots M_{p-1} \pmod p$ for all $p < N$.
In Section~\ref{sec:tree} we will explain how to efficiently compute products of this type;
this step constitutes the bulk of the running time of the main algorithm.

\section{The accumulating remainder tree over a quadratic field}
\label{sec:tree}

The accumulating remainder tree is a computational technique that lies at the heart of all of the recent average polynomial time point-counting algorithms.
The basic scalar version was introduced in \cite{Harvey:WilsonPrimes}, and it was generalized to integer matrices in \cite{Harvey:HyperellipticPolytime}.
In this section we present a variant that works over the ring of integers of a quadratic field $K$.

We will use the same notation as in \cite[Section~3]{HS-hassewitt2}.
Let $b \geq 2$ and $r \geq 1$.
Let $m_1, \ldots, m_{b-1}$ be a sequence of positive integers.
Let $A_0, \ldots, A_{b-2}$ be a sequence of $r \times r$ matrices with entries in $O_K$, and let $V$ be an $r$-dimensional row vector with entries in $O_K$.
The aim is to compute the sequence of reduced row vectors $C_1, \ldots, C_{b-1}$ defined by
\begin{equation}
\label{eq:Cn}
 C_n \coloneqq V A_0 \cdots A_{n-1} \bmod m_n.
\end{equation}
So far, this setup is identical to \cite{HS-hassewitt2}, except that in that paper $A_j$ and $C_j$ had entries in $\ZZ$ rather than $O_K$.

To apply this to the situation in Section \ref{sec:recurrences}, we set $r = 8$, $b = \lfloor N/2 \rfloor$, $A_j = M_{2j+1} M_{2j+2}$, $V = V_0$, and $m_n = 2n+1$ if $2n+1$ is prime, or $1$ if not.
Then for any odd $p < N$ we have $C_{(p-1)/2} = V_0 M_1 \cdots M_{p-1} \pmod p$, from which we can read off the entries of $U_p$ by Proposition~\ref{prop:Up} (provided that $(h_0, p) = 1$).

The naive algorithm for computing $C_n$, which separately computes each product $V A_0 \cdots A_{n-1}$ modulo $m_n$, leads to a running time bound that is quasi-quadratic in $b$.
The accumulating remainder tree improves this to a quasi-linear bound.
Pseudocode is given in Algorithm \textsc{QuadraticRemainderTree} below.
For simplicity we assume that $b = 2^\ell$ is a power of two, although this is not strictly necessary.
The algorithm actually computes various intermediate quantities $m_{i,j}$, $A_{i,j}$ and $C_{i,j}$, where $0 \leq i \leq \ell$ and $0 \leq j < 2^i$ (see \cite{HS-hassewitt2} for precise definitions);
the output is obtained as $C_j = C_{\ell,j}$.
For convenience we set $m_0 = 1$ and let $A_{b-1}$ be the identity matrix.

\bigskip

\noindent
\textbf{Algorithm} \textsc{QuadraticRemainderTree}
\vspace{2pt}

\noindent
Given $V, A_0,\ldots,A_{b-1}$, $m_0,\ldots,m_{b-1}$, with $b=2^\ell$, compute $m_{i,j}, A_{i,j}$, $C_{i,j}$:
\smallskip

\begin{enumerate}[1.]
\setlength{\itemsep}{2pt}
\item Set $m_{\ell,j}=m_j$ and $A_{\ell,j}=A_j$, for $0\le j < b$.
\item For $i$ from $\ell-1$ down to 0:\\
\phantom{For} For $0\le j < 2^i$, set $m_{i,j}=m_{i+1,2j}m_{i+1,2j+1}$ and $A_{i,j}=A_{i+1,2j}A_{i+1,2j+1}$.
\item Set $C_{0,0}=V \bmod m_{0,0}$ and then for $i$ from 1 to $\ell$:\\
\phantom{Set }For $0 \leq j < 2^i$ set $C_{i,j} =
\begin{cases}
C_{i-1,\lfloor j/2\rfloor}\bmod m_{i,j}\qquad&\text{if $j$ is even,}\\ 
C_{i-1,\lfloor j/2\rfloor}A_{i,j-1}\bmod m_{i,j}&\text{if $j$ is odd.}\\ 
\end{cases}$
\end{enumerate}
\bigskip

In fact, this pseudocode is copied verbatim from algorithm \textsc{RemainderTree} in \cite{HS-hassewitt2}.
The only difference between \textsc{RemainderTree} and \textsc{QuadraticRemainderTree} is the underlying data type; in \textsc{RemainderTree} the objects $A_{i,j}$ and $C_{i,j}$ are defined over $\ZZ$, whereas in \textsc{QuadraticRemainderTree} they are defined over $O_K$.
In all other respects, including the proof of correctness, the algorithms are identical.

The following theorem summarizes the performance characteristics of \textsc{QuadraticRemainderTree}.
The bit size of an element of $O_K = \ZZ[\alpha]$ is defined to be the maximum of the bit sizes of the coefficients of $1$ and $\alpha$.
\begin{theorem}
\label{thm:treebound}
Let $B$ be an upper bound for the bit size of $\prod_{j=0}^{b-1} m_j$, let $B'$ be an upper bound for the bit size of any entry of $V$, and let $H$ be an upper bound for the bit size of any $m_0, \ldots, m_{b-1}$ and any entry of $A_0, \ldots, A_{b-1}$.
Assume that $\log r = O(H)$ and that $r = O(\log b)$.
The running time of the \textsc{QuadraticRemainderTree} algorithm is
 \[ O(r^2 \MM(B + bH) \log b + r \MM(B')), \]
and its space complexity is $O(r^2 (B + bH) \log b + rB')$.
\end{theorem}
\begin{proof}
The statement is identical to Theorem 3.2 of \cite{HS-hassewitt2}.
The only difference in the analysis is that we must bound the cost of all operations over $O_K$ instead of over $\ZZ$.
We assume for this discussion that $D$ is \emph{fixed}; in our applications we arrange for $D$ to be small, say $-1$ or $2$.

The main operation to consider is computing the product of two $r \times r$ matrices, say $R$ and~$S$, with entries in $O_K$.
Write them as $R = R_0 + R_1 \alpha$ and $S = S_0 + S_1 \alpha$, where the $R_i$ and $S_i$ are integer matrices.
In the $D \not\equiv 1 \pmod 4$ case, we have $RS = (R_0 S_0 + D R_1 S_1) + (R_0 S_1 + R_0 S_1) \alpha$.
This clearly reduces to four matrix multiplications over $\ZZ$, plus several much cheaper operations (matrix additions, and scalar multiplication by $D$).
A similar formula holds for the $D \equiv 1 \pmod 4$ case, and similar remarks apply to the matrix-vector multiplications in step 3.

Overall, we clearly lose only a constant factor compared to the analysis in~\cite{HS-hassewitt2}.
\end{proof}

\begin{remark}
\label{rem:forest}
One can greatly improve the space consumption (and to a lesser extent, the running time) of the accumulating remainder tree algorithm by utilizing the \emph{remainder forest} technique introduced in \cite{HS:HyperellipticHasseWitt}; see also \cite[Theorem~3.3]{HS-hassewitt2}.
The idea is to split the work into $2^\kappa$ subtrees, where $\kappa \in [0,\ell]$ is a parameter.
This is important for practical computations, because \textsc{QuadraticRemainderTree} is extremely memory intensive.
\end{remark}

\subsection{Practical considerations}
\label{sec:practical}

In practice, the running time of the main algorithm is dominated by the matrix-matrix and matrix-vector multiplications over $\ZZ[\alpha]$, so it is important to optimize this step.

Let us first recall the discussion in \cite{HS-hassewitt2} for the case of matrices over $\ZZ$.
For multiplying $r \times r$ integer matrices, the classical matrix multiplication algorithm requires $r^3$ integer multiplications.
This is exactly what we do near the bottom of the tree, where the matrix entries are relatively small.

Further up the tree, when the matrix entries become sufficiently large, it becomes profitable to use FFT-based integer multiplication.
For example, the well-known GMP multiple-precision arithmetic library \cite{gmp-6.0} will automatically switch to a variant of the Sch\"onhage--Strassen algorithm for large enough multiplicands.
However, this is inefficient because each matrix entry will be transformed $r$ times.
This redundancy can be eliminated by means of the following alternative algorithm: (1) transform each of the $2r^2$ matrix entries, then (2) multiply the matrices of Fourier coefficients, and finally (3) perform an inverse transform on each of the $r^2$ entries of the target matrix.
This strategy reduces the number of transforms from $3r^3$ to $3r^2$.

Unfortunately, in our implementation we cannot carry out this plan using GMP, because GMP currently does not provide an interface to access the internals of its FFT representation.
Moreover, the Sch\"onhage--Strassen framework is not well suited to the matrix case, because the Fourier coefficients are relatively large.
Instead, we implemented our own FFT based on number-theoretic transforms modulo word-sized primes; see \cite[Section 5.1]{HS:HyperellipticHasseWitt}.

Asymptotically, for large matrix entries, we expect the running time to be dominated by the Fourier transforms, and so we expect a speedup of a factor of about $r$ compared to the classical algorithm.
The measured speedup is somewhat less than this, because of the contribution of step (2).
For example, taking $r = 8$ and matrices with entries of 500 million bits, we observe a speedup of around 6.4, rather than~8.

Turning now to $\ZZ[\alpha]$, the same principle applies.
Suppose that $D \not\equiv 1 \pmod 4$ and that $R = R_0 + R_1\alpha$ and $S = S_0 + S_1\alpha$, where the $R_i$ and $S_i$ are integer matrices.
We may write the product as
 $RS = (R_0 S_0 + (DR_1)S_1) + (R_0 S_1 + R_1 S_0) \alpha$.
We compute this as follows:
\begin{enumerate}
\item Transform the entries of $R_0$, $S_0$, $R_1$, $S_1$ and also $DR_1$.
There are $5r^2$ transforms here.
Denote these by $T(R_0), \ldots, T(DR_1)$.
\item Multiply the matrices of Fourier coefficients, to obtain $T(R_0) T(S_0)$, $T(DR_1) T(S_1)$, $T(R_0) T(S_1)$, and $T(R_1) T(S_0)$.
\item Add the matrices of Fourier coefficients, to obtain $T(R_0) T(S_0) + T(DR_1) T(S_1)$ and $T(R_0) T(S_1) + T(R_1) T(S_0)$.
\item Perform $2r^2$ inverse transforms to obtain the components of each entry of $RS$.
\end{enumerate}
A similar discussion applies to the $D \equiv 1 \pmod 4$ case.

Altogether we count $7r^2$ transforms, compared to $3r^2$ for the plain integer case.
We thus expect the ratio of the cost of multiplying two matrices over $\ZZ[\alpha]$ to the cost of multiplying two matrices over $\ZZ$ to be about $7/3 \approx 2.33$, assuming inputs of the same bit size.
The measured ratio is somewhat worse than this, mainly because of the non-negligible contribution of step~(2).
For example, with $r = 8$ and entries of 500 million bits, we observe a ratio of around 2.67.

One further optimization, which we did not pursue in our implementation, is to absorb the factor $D$ directly into the transforms themselves.
For example, if $D = -1$, the transform of $DR_1$ is just the negative of the transform of $R_1$, which we have already computed.
This would reduce the number of transforms from $7r^2$ to $6r^2$.
Unfortunately, this leads to technical complications for larger values of $|D|$, because the size of the Fourier coefficients needs to be increased to accommodate the extra factor of $D$.
In the context of our ``small prime'' FFTs, this optimization might be reasonable for very small $|D|$, but in the interests of maintaining generality and simplicity of our code, we did not implement it.

\section{Computing the Hasse--Witt matrices}
\label{sec:hw}

We now return to the hyperelliptic model $y^2 = h(x)$ for $C'$ over $K = \QQ(\sqrt D)$ that was constructed in Section \ref{sec:model}.

For each odd prime $p < N$ we select a prime ideal $\pp$ of $K$ above $p$ that we assume is unramified (we ignore the primes $p$ that ramify in $K$).
Now assume that $\pp$ does not divide the discriminant of $h(x)$, so that $C'$ has good reduction at $\pp$.
The \emph{Hasse--Witt matrix} of $C'$ at $\pp$ is the $3 \times 3$ matrix $W_\pp$ over $O_K/\pp$ with entries
 \[ (W_\pp)_{i,j} = h^{(p-1)/2}_{pi-j} \bmod \pp, \qquad i, j \in \{1, 2, 3\}. \]
There is a close relationship between $W_\pp$ and the local zeta function of the original curve $C$, which is discussed in Section \ref{sec:lpolys}.
In the remainder of this section we explain how to compute the~$W_\pp$.

Let $W^1_\pp$ denote the \emph{first row} of $W_\pp$.
By definition, $W^1_\pp$ is simply $U_p \pmod \pp$, where $U_p$ is the vector defined in Section \ref{sec:recurrences}.
We may therefore compute $W^1_\pp$ for all $p < N$ by using \textsc{QuadraticRemainderTree} (Section \ref{sec:tree}) to compute $U_p$ for all $p < N$, and then reduce each~$U_p$ modulo our chosen prime ideal~$\pp$ for each prime.

To obtain the remaining rows of $W_\pp$, the most obvious approach is to continue iterating the recurrence of Section~\ref{sec:recurrences} to reach $v_{2p-1}$ and $v_{3p-1}$ (in the notation of the proof of Proposition~\ref{prop:Up}).
This can be made to work, but there are technical difficulties: the factor of $k$ in the denominator of \eqref{eq:vk} leads to divisions by $p$.
Bostan, Gaudry and Schost deal with this by artificially introducing extra $p$-adic digits \cite{BGS-recurrences}.
We will use instead the following trick, which was suggested in \cite[Section~5]{HS-hassewitt2}.

For each integer $\beta$, let $W_\pp(\beta)$ denote the Hasse--Witt matrix of the \emph{translated} curve $y^2 = h(x+\beta)$, and let $W^1_\pp(\beta)$ denote its first row.
The relation between $W_\pp$ and $W_\pp(\beta)$ is given by
\begin{equation}
\label{eq:hw-translate}
  W_\pp(\beta) = T(\beta) W_\pp T(-\beta),
\end{equation}
where
 \[ T(\beta) = \begin{bmatrix} 1 & \beta & \beta^2 \\ 0 & 1 & 2\beta \\ 0 & 0 & 1 \end{bmatrix}. \]
For a proof, see \cite[Theorem~5.1]{HS-hassewitt2}.
(Note that in \cite{HS-hassewitt2} we work over $\FF_p$ whereas here we are possibly working over an extension, but this does not change the resulting formula, because $\beta$ is a rational integer.)

Now suppose that we have computed $W^1_\pp(\beta_i)$ for three integers $\beta_1, \beta_2, \beta_3$, and we wish to deduce $W_\pp$.
For each $i$, the equation $W_\pp(\beta_i) = T(\beta_i) W_\pp T(-\beta_i)$ yields a system of three linear equations in the nine unknown entries of $W_\pp$.
We therefore have nine equations in nine unknowns, and the same argument as in \cite[Section~5]{HS-hassewitt2} shows that this system has a unique solution, provided that $\beta_1$, $\beta_2$ and $\beta_3$ are distinct modulo $p$.

\section{Computing the $L$-polynomials modulo $p$}
\label{sec:lpolys}

At this stage, for each prime $p < N$ (except for various exceptional primes), we have computed $W_\pp$ for our chosen $\pp$ above $p$.
In this section we explain how this determines $L_p(T) \pmod p$ in the split case, and $L_p(T) L_p(-T) \pmod p$ in the inert case.

Consider the zeta function of $C'$ at $\pp$.
This is defined by
 \[ Z'_\pp(T) \coloneqq \exp\left(\sum_{k=1}^\infty \frac{N_k}k T^k \right) = \frac{L'_\pp(T)}{(1-T)(1-qT)}, \]
where $N_k$ is the number of points on $C'_\pp$ (the reduction of $C'$ modulo $\pp$) defined over the extension of $O_K/\pp$ of degree $k$.
As before, $L'_\pp(T) \in \ZZ[T]$ has degree six.

In the split case, we simply have
 $L'_\pp(T) = \det(I - T W_\pp) \pmod{\pp}$.
Thus $W_\pp$ determines $L'_\pp(T) \pmod p$.
Moreover, since $C'_\pp$ is isomorphic to $C_p$ over $O_K/\pp \cong \ZZ/p\ZZ$, they have the same zeta functions, so
 $L_p(T) = L'_\pp(T)$ (in $\ZZ[T]$).
Hence $W_\pp$ determines $L_p(T) \pmod p$.

In the inert case, we have
 $L'_\pp(T) = \det(I - T W_\pp W_\pp^{(p)}) \pmod{\pp}$,
where $W_\pp^{(p)}$ denotes the matrix obtained by applying the absolute Frobenius map to each entry of $W_\pp$, which raises each entry to the $p$-th power.
So again in this case $W_\pp$ determines $L_\pp(T) \pmod p$.
Unfortunately, because of the base change from $\QQ$ to $K$, we lose information when passing from $C$ to $C'$; in effect, we have computed the zeta function of $C_p$ over $\FF_{p^2}$.
All we can conclude is that
 $L_p(T) L_p(-T) = L'_\pp(T^2)$
(see \cite[Ch.~VIII, Lemma~5.12]{Lor-invitation}), so $W_\pp$ determines only $L_p(T) L_p(-T) \pmod p$.

\section{Lifting the $L$-polynomials}
\label{sec:lifting}

We now turn to the problem of determining $L_p(T)\in \ZZ[T]$, given as input either (1) $L_p(T) \pmod p$ (for $p$ split in $K$), or (2) $L_p(T) L_p(-T) \pmod p$ (for $p$ inert in $K$).
Our approach to this problem utilizes generic group algorithms operating in $\Jac(C_p)(\FF_p)$, the group of $\FF_p$-rational points on the Jacobian variety of the reduction of $C$ modulo $p$.
It is a finite abelian group of order $p^3+O(p^{5/2})$.

We first need a model for the curve that supports efficient arithmetic in $\Jac(C_p)(\FF_p)$.
We start with the reduction modulo $p$ of the model given in \eqref{eq:eqn}.
Although the conic $g=0$ has no $\QQ$-rational points, its reduction modulo $p$ does have $\FF_p$-rational points, and therefore admits a rational parametrization that can be used to construct a hyperelliptic model $y^2=h(x)$ with $h \in \FF_p[x]$, as in Section \ref{sec:model}.
The cost of constructing this model is negligible.
Now, if $C_p$ has a rational Weierstrass point, we move it to infinity and thus make $h(x)$ monic of degree~7; in this case fast explicit formulas for arithmetic in $\Jac(C_p)(\FF_p)$ are well known \cite[\S 14.6]{HECECC}.
If $C_p$ does not have a rational Weierstrass point, then provided $p\ge 37$ (which we assume), it has a rational non-Weierstrass point $P$; moving this point to infinity, we obtain a model with $h(x)$ monic of degree~8.
Fast explicit formulas for arithmetic in $\Jac(C_p)(\FF_p)$ for such models have recently been developed~\cite{Sut:Genus3Real}, using the balanced divisor approach of \cite{GHM:BalancedDivisors,Mor:Thesis}.

Case (1) is considered in \cite{KS:HyperellipticLSeries}, where it is noted that the problem of determining $L_p(T)\in \ZZ[T]$ given $L_p(T)\pmod p$ can be solved in $p^{1/4+o(1)}$ time (for a curve of genus $3$).
Let us briefly recall how this is done.

If $p\ge 149$, then $L_p(T)\pmod p$ uniquely determines the coefficient~$a_1$ of $L_p(T)$.
Indeed, from the Weil bounds we have $|a_i|\le\binom{6}{i}p^{i/2}$ for $i = 1, 2, 3$.
This inequality constrains $a_2$ to at most $2\binom{6}{2}=30$ values compatible with $a_2\pmod p$.
In fact, once~$a_1$ is known, there are at most 6 possibilities for $a_2$; this follows from \cite[Prop. 4]{KS:HyperellipticLSeries}.
For each of these 6 values of $a_2$, the pair $(a_1,a_2)$ determines a set of at most $40p^{1/2}$ possible values of $a_3$, corresponding to an arithmetic progression modulo $p$.
The pair $(a_1,a_2)$ also determines corresponding arithmetic progressions modulo $p$ in which the integers
\begin{align}\label{eq:jacorders}
\#\Jac(C_p)(\FF_p) = L_p(1)&=(p^3+1)+(p^2+1)a_1+(p+1)a_2+a_3,\\\notag
\#\Jac(\tilde{C}_p)(\FF_p) = L_p(-1)&=(p^3+1)-(p^2+1)a_1+(p+1)a_2-a_3
\end{align}
must lie; here $\tilde{C}_p$ denotes a (non-trivial) quadratic twist of $C_p$.

Now, given any $\alpha \in \Jac(C_p)(\FF_p)$ (or $\Jac(\tilde{C}_p)(\FF_p)$), we may compute its order $|\alpha|$ as follows.
First, apply a baby-steps giant-steps search to the appropriate arithmetic progression to obtain a multiple~$m$ of $|\alpha|$.
Then factor $m$ and use a polynomial-time fast order algorithm (see \cite[Ch.~7]{Sut:Thesis}) to compute $|\alpha|$.
The time to factor $m=O(p^3)$ is negligible compared to the cost of the baby-steps giant-steps search, both in theory \cite{LP:Factoring} and in practice.
Note that if our candidate value of~$a_2$ is incorrect, we may not find such an $m$, in which case we discard this value of~$a_2$ and proceed to the next of our (at most 6) candidates.
One of the candidates must work, hence we can determine the order of $\alpha$ in $p^{1/4+o(1)}$ time.
This applies more generally to any situation where we have $O(1)$ possible pairs $(a_1,a_2)$ and we know the value of $a_3$ modulo~$p$; this includes case~(2), as we explain below (and also the case $p\le 149$).

With the ability to compute the orders of arbitrary group elements, we obtain a Monte Carlo algorithm to compute the group exponent $\lambda$ of $\Jac(C_p)(\FF_p)$ in $p^{1/4+o(1)}$ time via \cite[Alg. 8.1]{Sut:Thesis} (and similarly for $\Jac(\tilde C_p)(\FF_p)$).
The positive integer $n$ output by this algorithm is guaranteed to divide~$\lambda$, and the probability that $n\ne \lambda$ can be made arbitrarily small, at an exponential rate.
Note that this algorithm needs access to random elements of $\Jac(C_p)(\FF_p)$; such elements may be found by picking random polynomials $u\in \FF_p[x]$ with $\deg u\le g=3$ and attempting to construct the Mumford representation $[u(x),v(x)]$ of the affine part of a representative for a divisor class in $\Jac(C_p)(\FF_p)$.
This can be viewed as a generalization of the decompression technique described in \cite[\S14.2]{HECECC}.

As shown in \cite[Prop. 4]{KS:HyperellipticLSeries}, given the group exponent $\lambda$, we can compute $\#\Jac(C_p)(\FF_p)$ using the generic group algorithm in \cite[Alg. 9.1]{Sut:Thesis} in $p^{1/4+o(1)}$ time.
The same applies to $\#\Jac(\tilde{C}_p)(\FF_p)$; we can thus determine the values of both $L_p(1)$ and $L_p(-1)$, which suffice to determine $L_p(T)$.
Indeed, adding the equations in \eqref{eq:jacorders} yields the value of $a_2$, and subtracting them and substituting $a_1$ yields $a_3$ (see \cite[Lemma 4]{Sut:GenericJacobians} for a more general result that applies whenever $p\ge 1600$).

The fact that we used a Monte Carlo algorithm to compute $\lambda$ means that there is some (exponentially small) probability of error.
We can eliminate this possibility by considering the set $S$ of candidate values for $\#\Jac(C_p)$ that are both multiples of our divisor~$n$ of $\lambda$ and compatible with the constraints imposed by \eqref{eq:jacorders}, the set of candidate pairs $(a_1,a_2)$, and the value of $a_3$ modulo $p$.
Typically $|S| = 1$ and we immediately obtain a verified result.
If not, any two candidates $N_1$ and $N_2$ for $\#\Jac(C_p)$ must differ in their $\ell$-adic valuations for at least two primes $\ell$ (for $p>30$ we cannot have $N_1$ divisible by $N_2$ or vice versa).
By computing the group structure of the $\ell$-Sylow subgroup $H$ of the smaller of these two primes $\ell$ via \cite[Alg. 9.1]{Sut:Thesis} (a Monte Carlo algorithm that always outputs a subgroup of $H$), we may be able to provably rule out one of the candidates by obtaining a lower bound on the $\ell$-adic valuation of $\#\Jac(C_p)(\FF_p)$ that exceeds the $\ell$-adic valuation of one of them.
Provided $\ell=O(p^{1/2})$, this takes $p^{1/4+o(1)}$ time; we can also use $\#\Jac(\tilde C_p)$.
In the computations described in Section \ref{sec:performance} this method was used to verify $L_p(T)$ in every case; we expect that one can prove that the complexity of this computation is bounded by $p^{1/4+o(1)}$ (at least on average), but we do not attempt this here.
The timings listed in Table~\ref{table:comparison} include the (negligible) cost of this verification in the ``lift'' columns.

In case (2), where we are given $L_p(T)L_p(-T)\pmod p $, there are at most 8 possible values of $L_p(T)\pmod p$.
To see this, let $\sum_{i=0}^6 b_i T^{2i}=L_p(T)L_p(-T)$. One obtains the relations
\[
b_1 \equiv 2a_2-a_1^2 \pmod p , \qquad b_2 \equiv a_2^2-2a_1a_3\pmod p, \qquad b_3 \equiv -a_3^2\pmod p.
\]
Given $b_1,b_2,b_3\pmod p$, there are two possibilities for $a_3\pmod p$, each of which determines a pair of quadratic equations in $a_1$ and $a_2$, which in turn has at most four solutions modulo~$p$.
Even though the value of $a_1$ is not uniquely determined in this case (no matter how big $p$ is), we can apply the procedure described above to compute the orders of arbitrary elements of $\Jac(C_p)(\FF_p)$ or $\Jac(\tilde C_p)(\FF_p)$ in $p^{1/4+o(1)}$ time, and the rest of the discussion follows; the key point is that we have $O(1)$ arithmetic progressions of length $O(p^{1/2})$ in which $\#\Jac(C_p)(\FF_p)$ and $\#\Jac(\tilde C_p)(\FF_p)$ are known to lie.

\section{Summary of the algorithm}
\label{sec:summary}

We now describe the complete algorithm.
The input consists of the polynomials $f$ and $g$ defining the curve $C$ according to \eqref{eq:eqn}, a bound $N$, and a parameter~$\kappa$ (see Remark \ref{rem:forest}).
Our goal is to compute $L_p(T) \in \ZZ[T]$ for all odd primes $p < N$, except for a small number of exceptional primes as documented below.

\bigskip

\begin{enumerate}[1.]
\item Find a quadratic field $K = \QQ(\sqrt{D})$ and a suitable model $y^2 = h(x)$ for $C'$ over $K$, using (for example) the method of Section \ref{sec:model}. \\
Choose small integers $\beta_1, \beta_2, \beta_3$ so that $h(x+\beta_i) \neq 0$ for each $i$.
\item Make a list of all odd primes $p < N$. For each $p$:
\begin{itemize}
\item If $p$ satisfies any of the following conditions, declare $p$ exceptional:
\begin{itemize}
\item $p$ divides $D$  (ramified prime).
\item $p$ divides some $\beta_i - \beta_j$.
\item $p$ is not relatively prime to the discriminant of $h(x)$ (and hence of $h(x+\beta_i)$ for all $i$).
\item $p$ is not relatively prime to the constant term of some $h(x+\beta_i)$.
\end{itemize}
\item Otherwise:
\begin{itemize}
\item If $(D/p) = 1$ (split prime), pick a solution of $\gamma^2 = D \pmod p$ and let $\pp = (p, \gamma - \sqrt D)$ be the corresponding prime ideal above $p$.
\item If $(D/p) = -1$ (inert prime), let $\pp = (p)$.
\end{itemize}
\end{itemize}
\item Let $U_p(\beta_i)$ be the vector $U_p$ (defined in Section \ref{sec:recurrences}) corresponding to the translated curve $y^2 = h(x+\beta_i)$. Call \textsc{QuadraticRemainderTree} (or the ``forest'' variant with parameter~$\kappa$) three times, once for each translated curve, with parameters as specified in Section \ref{sec:tree}, to compute $U_p(\beta_i)$ for all non-exceptional $p < N$.
\item For each non-exceptional prime $p < N$:
\begin{itemize}
\item Reduce $U_p(\beta_i)$ modulo $\pp$ to obtain $W^1_\pp(\beta_i)$ for $i = 1, 2, 3$.
\item Solve the system described at the end of Section \ref{sec:hw}, using \eqref{eq:hw-translate} to deduce $W_\pp$.
\item Compute $\det(I - T W_\pp)$ (split case) or $\det(I - T W_\pp W^{(p)}_\pp)$ (inert case), to determine $L_p(T) \pmod p$ (split case) or $L_p(T) L_p(-T) \pmod p$ (inert case), according to Section \ref{sec:lpolys}.
\item Apply the lifting procedure of Section \ref{sec:lifting} to finally obtain $L_p(T) \in \ZZ[T]$.
\end{itemize}
\end{enumerate}

\medskip

As pointed out earlier, in practice the running time is dominated by the calls to \textsc{QuadraticRemainderTree}.
The exceptional primes (of good reduction) can be handled by any other suitable method; for example, naive point counting for the small exceptional primes, and for the larger ones, parametrizing the conic over $\FF_p$ and then applying \cite{Har-kedlaya}.
One can easily prove that the number of exceptional primes is small, and one can also prove that these primes make negligible overall contribution to the complexity.
We omit the details.

\section{Implementation and performance}
\label{sec:performance}

We implemented most of the steps of the main algorithm in the C programming language, building on the implementation for the ordinary hyperelliptic case described in \cite{HS-hassewitt2}.
It uses the GMP library \cite{gmp-6.0} for basic integer arithmetic, and a customized FFT library for matrix arithmetic over $O_K$ when the entries have large coefficients; see Section \ref{sec:practical} and \cite[Section 5.1]{HS:HyperellipticHasseWitt}.

The program takes as input the original model for the curve $C$ over $\QQ$, and also the data describing the model $C'$ over $K$, namely the integer $D$ and the polynomial $h \in O_K[x]$.
The construction of $C'$ itself is not yet fully automated; for this we use \emph{ad hoc} methods, including Magma \cite{magma} and Sage \cite{sage-6.8} scripts.
The output is the sequence of polynomials $L_p(T)$ for all $p < N$, except for a small number of exceptional primes (listed below).
As pointed out earlier, it is not difficult to handle the missing primes of good reduction, but our implementation does not yet do this.
For the application to the Sato--Tate problem, it is safe to ignore a small number of primes, because they have a negligible effect on the statistical data being collected, but for the application to computing $L$-series one would need to address the missing primes (including those of bad reduction, a problem we do not address here).

We give one numerical example to illustrate the performance of our implementation, and compare it to the implementation for the ordinary hyperelliptic case from \cite{HS-hassewitt2}.
For the hyperelliptic case we take the curve $C_1$ defined by
\[
y^2 = 2x^8 - 2x^7 + 3x^6 - 2x^5 - 4x^4 + 2x^3 + 2x + 2.
\]
For the new algorithm, we take $C_2$ to be the curve given by $w^2 = f(X,Y,Z)$ where
\begin{multline*}
 f(X,Y,Z) = X^4 - 2 X^2 Y^2 - 2 Y^4 - X^3 Z - 2 X^2 Y Z - X Y^2 Z \\
   - Y^3 Z - X^2 Z^2 - X Y Z^2 - Y^2 Z^2 + X Z^3 + Z^4,
\end{multline*}
over the pointless conic $X^2 + Y^2 + Z^2 = 0$.
We base extend to $K = \QQ(i)$ with $i^2 = -1$ (so $D = -1$), and we parametrize the conic by
 $(\psi_1(x), \psi_2(x), \psi_3(x)) = (x^2 - 1, 2u, i(x^2 + 1))$.
This leads to the curve $C'_2$ over $K$ given by the hyperelliptic equation $y^2 = h(x)$ where
\begin{multline*}
 h(x) = (3 - 2i)x^8 + (2 - 4i)x^7 + (-4-4i)x^6 + (2-4i)x^5 + \\
   2x^4 + (-2-4i)x^3 + (-4+4i)x^2 + (-2-4i)x + (3 + 2i).
\end{multline*}
Note that the polynomial $f(X,Y,Z)$ was chosen carefully (by a random search) to ensure that the coefficients of $h(x)$ would not be too large.

We ran both programs to determine the zeta functions for $C_1$ and $C_2$ at all $p < N$ for various values of $N$.
In both cases we used the translates $\beta_i = i$ for $i = 0, 1, 2$.
For $C_1$ the exceptional primes were
3, 5, 7, 19, 181, 931781;
for $C_2$ they were 3, 5, 7, 13, 31, 269, 10169, 22229.
The computations were run on a single core of an otherwise idle 64-core 2.5 GHz Intel Xeon (E7-8867W v3) server with 1088 GB RAM, running Ubuntu Linux version 14.04.
We used the GCC compiler, version 4.8.4 \cite{gcc-4.8.4}, with optimization flags \texttt{-O3 -funroll-loops}.

Performance figures are given in Table \ref{table:comparison}.
We set the parameter $\kappa$ (see Remark~\ref{rem:forest}) to $7$ in all our tests, a choice that optimized (or very nearly optimized) the running time in every case.
The ``time'' columns show the total running time, excluding the lifting phase, and the ``space'' columns show the peak memory usage.
The running time of the lifting phase is given in the ``lift'' column; the memory usage is negligible for this phase.

The last two columns give estimates for the time to run \texttt{hypellfrob}, an implementation of the algorithm described in \cite{Har-kedlaya}.
For a hyperelliptic curve of genus $3$, and for a given $p$-adic precision parameter $\alpha \geq 1$, it computes $L_p(T) \pmod{p^\alpha}$ in time $\alpha^{O(1)} p^{1/2+o(1)}$ for each $p$ separately; prior to \cite{HS:HyperellipticHasseWitt,HS-hassewitt2}, it was the fastest available software for this problem.
The ``mod $p^\alpha$'' column, for $\alpha = 1, 2$, gives an estimate for the total time to compute $L_p(T) \pmod{p^\alpha}$ for all $p < N$.
The estimates were obtained by sampling for several $p < N$ and extrapolating based on the number of primes in each interval.
For $\alpha = 1$, this is enough to determine $a_1$ (provided $p \geq 149$) but not all of $L_p(T)$; one would still need to run a lifting step to obtain $L_p(T)$.
For $\alpha = 2$, it determines $L_p(T)$ completely.

Note that \texttt{hypellfrob} is limited to curves with a rational Weierstrass point, so we used the curve $y^2 = x^7 + 3x^6 + 5x^5 + 7x^4 + 11x^3 + 13x^2 + 17x + 19$.
For this reason the timings are not directly comparable to the columns for $C_1$ and $C_2$, but they still provide a reasonable indication of what should be expected.
No implementation for the general case $y^2 = h(x)$ with $\deg h = 8$ is currently available; one could presumably be developed by adapting \cite{Har-extension}.

It is clear from Table \ref{table:comparison} that, broadly speaking, the new algorithm performs similarly to its hyperelliptic antecedent \cite{HS-hassewitt2}.
In particular, the running time is close to linear in $N$.
For the largest $N$ in the table, we observe a slowdown from $C_1$ to $C_2$ of a factor of about $3$.
This is only slightly worse than the factor $2.33$ that one expects asymptotically (see Section \ref{sec:practical}).
For $N = 2^{30}$ we see that the new algorithm is nearly $50$ times faster than \texttt{hypellfrob}.
As promised, the lifting phase makes a negligible overall contribution to the running time.

The memory footprint for $C_2$ is about twice that for $C_1$.
This is exactly as expected, since the input coefficient sizes are roughly equal, and for $C_2$ we carry around twice as much information in each matrix (the coefficients of $1$ and $\alpha$).

An obvious disadvantage of the new algorithm is that it is more difficult to parallelize than \texttt{hypellfrob}.
The latter is trivially parallelizable, by distributing primes among threads.
In fact, there is some scope for parallelization in the new algorithm, but this is a rather involved question that will be deferred to a subsequent paper.

\begin{center}
\begin{table}[!htb]
\setlength{\tabcolsep}{6pt}
\begin{tabular}{@{}rrrrrrrrrrrrr@{}}
&&\multicolumn{3}{c}{$C_1$}&&\multicolumn{3}{c}{$C_2$}&&\multicolumn{2}{c}{\texttt{hypellfrob}}\\
\cmidrule(r){3-5}\cmidrule(r){7-9}\cmidrule(r){11-12}
$N$      &&     time &            space &     lift &&      time &            space &     lift &&         mod $p$ & mod $p^2$ \\
\midrule
$2^{16}$ &&        4 &   0.05           &        2 &&            14 &   0.06           &        3 &&             36  &                 127 \\
$2^{17}$ &&        9 &   0.06           &        4 &&            33 &   0.08           &        6 &&             92  &                 326 \\
$2^{18}$ &&       22 &   0.08           &        8 &&            75 &   0.11           &       13 &&            234  &                 849 \\
$2^{19}$ &&       53 &   0.10           &       16 &&           178 &   0.17           &       25 &&            600  &             2{,}680 \\
$2^{20}$ &&      129 &   0.17           &       32 &&           418 &   0.30           &       48 &&        1{,}770  &             7{,}500 \\
$2^{21}$ &&      310 &   0.30           &       66 &&           992 &   0.57           &       99 &&        4{,}830  &            25{,}300 \\
$2^{22}$ &&      753 &   0.58           &      136 &&       2{,}390 &   1.18           &      201 &&       14{,}900  &           189{,}000 \\
$2^{23}$ &&  1{,}780 &   1.13           &      278 &&       5{,}520 &   2.47           &      413 &&       42{,}700  &           653{,}000 \\
$2^{24}$ &&  4{,}090 &   2.41           &      574 &&      12{,}600 &   5.33           &      850 &&      125{,}000  &       1{,}680{,}000 \\
$2^{25}$ &&  9{,}410 &   4.98           &  1{,}190 &&      29{,}000 &  11.8\phantom{0} &  1{,}760 &&      395{,}000  &       5{,}030{,}000 \\ 
$2^{26}$ && 22{,}100 &  10.5\phantom{0} &  2{,}470 &&      66{,}300 &  24.5\phantom{0} &  3{,}650 &&  1{,}230{,}000  &      16{,}000{,}000 \\
$2^{27}$ && 50{,}900 &  23.5\phantom{0} &  5{,}160 &&     151{,}000 &  52.0\phantom{0} &  7{,}610 &&  3{,}730{,}000  &      44{,}100{,}000 \\
$2^{28}$ &&118{,}000 &  54.0\phantom{0} & 10{,}800 &&     344{,}000 & 112\phantom{.00} & 15{,}900 &&  10{,}000{,}000  &     113{,}000{,}000 \\
$2^{29}$ &&276{,}000 & 124\phantom{.00} & 22{,}800 &&     783{,}000 & 241\phantom{.00} & 33{,}600 &&  35{,}600{,}000  &     368{,}000{,}000 \\
$2^{30}$ &&681{,}000 & 288\phantom{.00} & 48{,}200 && 1{,}980{,}000 & 480\phantom{.00} & 71{,}100 &&  97{,}100{,}000  &     948{,}000{,}000 \\\bottomrule
\end{tabular}
\bigskip

\caption{Comparison of algorithms for computing $L_p(T)$ for $p < N$. See text for column explanations. Time in CPU seconds, space in gigabytes, all values rounded to three significant figures.}
\label{table:comparison}
\end{table}
\end{center}

\begin{acknowledgements}\label{ackref}
The authors thank Jesse Kass, Kiran Kedlaya, Christophe Ritzenthaler and John Voight for helpful conversations, and the referees for their comments that led to improvements in the presentation of these results.
\end{acknowledgements}

\bibliographystyle{amsplain}
\bibliography{pointless}

\affiliationone{
   David Harvey\\
	 School of Mathematics and Statistics\\
     University of New South Wales\\
	 Sydney NSW  2052\\
	 Australia
   \email{d.harvey@unsw.edu.au}}
\affiliationtwo{
   Maike Massierer\\
   School of Mathematics and Statistics\\
     University of New South Wales\\
	 Sydney NSW  2052\\
	 Australia
   \email{maike@unsw.edu.au}}
\affiliationthree{
    Andrew V. Sutherland\\
    Department of Mathematics\\
    Massachusetts Institute of Technology\\
    Cambridge, MA \ 02139, USA
    \email{drew@math.mit.edu}}   
    
\end{document}